\documentclass{amsart}
\pdfoutput=1
\usepackage{graphicx}
\newtheorem{theorem}{Theorem}[section]
\newtheorem{lemma}[theorem]{Lemma}

\theoremstyle{definition}

\newtheorem{example}[theorem]{Example}

\theoremstyle{remark}

\newtheorem{corollary}{Corollary}

\numberwithin{equation}{section}

\begin{document}

\title[System of two singularly perturbed convection-diffusion equations]{A parameter uniform fitted mesh method for a weakly coupled system of two singularly perturbed convection-diffusion equations}

\author{Saravana Sankar Kalaiselvan}
\address{Department of Mathematics, Bishop Heber College, Tiruchirappalli, Tamil Nadu, India}
\email{saravanasankar55@gmail.com}
\thanks{The first author has been supported with Junior Research Fellowship by UGC, India.}

\author{John J.H. Miller}
\address{Trinity College, Dublin, Ireland}
\email{jmiller@tcd.ie}

\author{Valarmathi Sigamani}
\address{Department of Mathematics, Bishop Heber College, Tiruchirappalli, Tamil Nadu, India}
\email{valarmathi07@gmail.com}





\begin{abstract}
In this paper, a boundary value problem for a singularly perturbed linear system of two second order ordinary differential equations of convection-diffusion type is considered on the interval $[0,1]$. The components of the solution of this system exhibit boundary layers at $0$. A numerical method composed of an upwind finite difference scheme applied on a piecewise uniform Shishkin mesh is suggested to solve the problem. The method is proved to be first order convergent in the maximum norm uniformly in the perturbation parameters. Numerical examples are provided in support of the theory.
\end{abstract}

\maketitle




\section{Introduction}
\label{sec:1}
Singular perturbation problems of convection-diffusion type arise in many areas of applied mathematics such as fluid dynamics, chemical reactor theory, etc. Also, linearising Navier-Stokes equations, which plays vital role in the field of science, leads to a system of convection-diffusion equations.\\
\indent	For a broad introduction to singularly perturbed boundary value problems of convection-diffusion type one can refer to  \cite{MOS96}, \cite{DMS80} and  \cite{FHM00}. There, the authors suggest robust computational techniques to solve them. A class of systems of singularly perturbed reaction-diffusion equations has been examined by several authors in \cite{NM03}, \cite{LM09}, \cite{PVM10} and \cite{CGL10}.\\
\indent Here, in this paper, a weakly coupled system of two singularly perturbed convection - diffusion equations with distinct perturbation parameters is studied both analytically and numerically. If the perturbation parameters are equal, then the arguments in \cite{FHM00} are sufficient to show that the suggested method is parameter uniform. But in general boundary layers of unequal width are expected for the components of the solution because of the coupling of the components.\\
\indent In the papers \cite{BO04} and \cite{OM09}, a class of strongly coupled systems of singularly perturbed convection-diffusion problems is examined. A coupled system of two singularly perturbed convection-diffusion equations is considered in \cite{ZC04}. In \cite{L07}, the author analysed  a coupled system of singularly perturbed convection-diffusion equations.
\\ \indent In this paper, the major assumptions $\varepsilon_1 \le CN^{-1},\; \varepsilon_2 \le CN^{-1}$ in \cite{ZC04}, are removed. Moreover the analytical and numerical arguments are completely different from \cite{ZC04} and \cite{L07} in the following sense. The decomposition of the solution is based on the effect of each perturbation parameter on the components of the solution. Thus, we get more information about the components of the solution and its layer pattern.  Also, it is to be noted that the decomposition of the smooth component in \cite{ZC04} is given a correct definition, here in this paper.\\
\\{\bf Notations.} For any  real valued function $y$ on $D$, the norm of $y$ is defined as $\|y\|_D= \displaystyle{\sup_{x \in D}}|y(x)| $. For any vector valued function $\vec{z}(x)=(z_1(x),z_2(x))^T$, $|\vec{z}(x)|=\big(|z_1(x)|,|z_2(x)|\big)^T$, $(\vec{z}(x))_i=z_i(x)$ and $\|\vec{z}\|_D=max\big\{\|z_1\|_D, \|z_2\|_D\big\}$. Also $\vec{z}(x) \ge \vec{0}$, if $z_1(x)\ge 0$ and $z_2(x)\ge 0$. 
 
 For any mesh function $Y$ on $D^N=\big\{x_j\big\}^N_{j=0}$, $\|Y\|_{D^N}= \displaystyle{\max_{0\le j\le N}}|Y(x_j)| $ and for any vector valued mesh function $\vec{Z}=(Z_1,Z_2)^T$, $|\vec{Z}(x_j)|=(|Z_1(x_j)|,|Z_2(x_j)|)^T$, $\|\vec{Z}\|_{D^N}=max\big\{\|Z_1\|_{D^N}, \|Z_2\|_{D^N}\big\}$.\\
\indent Throughout this paper, C denotes a generic positive constant which is independent of the singular perturbation and discretization parameters.

\section{Formulation of the problem}
\label{sec:2}
\noindent Consider the following system of equations
\begin{eqnarray}\label{e1}
L \vec{u}(x)  \equiv E\vec{u}^{\prime \prime}(x)+A(x)\vec{u}^{\prime}(x)-B(x)\vec{u}(x)=\vec{f}(x), x \in \Omega
\\\label{e2}
\vec{u}(0)=\vec{l},\;\; \vec{u}(1)=\vec{r},\hspace{5.6cm}
\end{eqnarray}
where, $\Omega= (0,1),\;\vec{u}(x)=(u_1(x),u_2(x))^T,\;\vec{f}(x)=(f_1(x),f_2(x))^T$,\;\;
$$ E =
\begin{bmatrix}
\varepsilon_1 & 0 \\ 0 & \varepsilon_2
\end{bmatrix},\;\;
A(x)=
\begin{bmatrix}
a_1(x) & 0 \\ 0 & a_2(x)
\end{bmatrix},\;\;
B(x)=
\begin{bmatrix}
b_{11}(x) & -b_{12}(x) \\ -b_{21}(x) & b_{22}(x)
\end{bmatrix}.
$$

Here, $ \varepsilon_1$ and $\varepsilon_2$ are two distinct small positive parameters and, without loss of generality, we assume that $\varepsilon_1 < \varepsilon_2.$ The coefficient functions are taken to be sufficiently smooth on $\overline{\Omega}$ and $a_i(x)\ge \alpha > 0,\; b_{ii}(x)-b_{ij}(x)\ge\beta>0,\; b_{ij}>0,\;$ for $i,j=1,2$ and $i\ne j.$\\
\indent The case $a_i(x)\le \alpha < 0$, for $i=1,2$, is put into the form (\ref{e1}) by the change of independent variable from $x$ to $1-x$.\\
\indent
Since, the matrix B(x) is not diagonal and the matrix A(x) is diagonal, the sytem is weakly coupled. If the matrix A(x) is not diagonal, then the system becomes strongly coupled. If $a_1(x)$ and $a_2(x)$ are zero functions, then the above problem comes under the class considered in \cite{NM03}.\\
\\
The reduced problem corresponding to (\ref{e1})-(\ref{e2})  is
\begin{eqnarray}
L_0 \vec{u_0}(x)  \equiv A(x)\vec{u_0}^{\prime}(x)-B(x)\vec{u_0}(x)=\vec{f}(x),\; x \in \Omega\\
 \vec{u_0}(1)=\vec{r},\hspace{6.1cm}
\end{eqnarray}
where, $\vec{u_0}(x)=(u_{01}(x),u_{02}(x))^T.$\\

A boundary layer of width $O(\varepsilon_2)$ is expected near $x=0$ in the solution components $u_1$ and $u_2$, if $u_2(0)\ne u_{02}(0)$ and
a boundary layer of width $O(\varepsilon_1)$ is expected near $x=0$ in the solution component $u_1$, if $u_1(0)\ne u_{01}(0)$. Numerical illustrations provided for each case exhibit such layer patterns.

\section{Analytical Results}
\label{sec:3}
In this section, a maximum principle, a stability result and estimates of the derivatives of the solution of the system of equations (\ref{e1})-(\ref{e2}) are presented.
\begin{lemma}[Maximum Principle]\label{L1}
Let  $\vec{\psi} \in (C^2({\overline\Omega}))^2$ such that $\vec{\psi}(0) \ge \vec{0},\;\vec{\psi}(1) \ge \vec{0}, \;L\vec{\psi} \le \vec{0}\; on\; (0,1)\;,  then \;\vec{\psi} \ge \vec{0} \; on\; [0,1].$
\end{lemma}
\begin{proof} 
Let $x^*$ and $y^*$ be such that $\psi_1(x^*)=\displaystyle{\min_{x\in\overline{\Omega}}} \psi_1(x)$ and $\psi_2(y^*)=\displaystyle{\min_{x\in\overline{\Omega}}} \psi_2(x)$. Without loss of generality, we assume that $\psi_1(x^*) \le \psi_2(y^*)$ and suppose $\psi_1(x^*)<0;$\;then $x^*\not\in\{0,1\},\; \psi_1^\prime(x^*)=0 \; \text{and}\; \psi_1^{\prime\prime}(x^*)\ge 0.$\\
$(L\vec{\psi})_1(x^*)\ge \varepsilon_1 \psi_1^{\prime \prime}(x^*) + a_1(x^*)\psi_1^\prime (x^*)-(b_{11}(x^*)-b_{12}(x^*))\psi_1(x^*)> 0,$ contradiction to the assumption that $ L\vec{\psi} \le \vec{0} $ on $(0,1)$.
Hence, $\vec{\psi}(x)\ge0$, on $[0,1]$. 
\end{proof}

An immediate consequence of the maximum principle is the following stability result.
\begin{lemma}[Stability Result]\label{L2}
Let $\vec{\psi} \in (C^2(\overline\Omega))^2$, then for $ x \in\overline\Omega $ and i=1,2
$$|\psi_i(x)|\le \text{max}\Big\{\|\vec{\psi}(0)\|,\; \|\vec{\psi}(1)\|,\; \frac{1}{ \beta } \|L\vec{\psi} \| \Big\}.$$
\end{lemma}
\begin{corollary} Let $\vec{u}$ be the solution of $(\ref{e1})-(\ref{e2})$, then 
$$|u_i(x)| \le \text{max}\Big\{\|\vec{l}\|,\; \|\vec{r}\|,\; \frac{1}{ \beta } \|\vec{f} \| \Big\}.$$
\end{corollary}
\begin{theorem}\label{T1}
Let $ {\vec u}$  be the solution of (\ref{e1})-(\ref{e2}), then for x $\in  \overline\Omega$ and i=1,2
\begin{eqnarray}\label{e5}
 |u_{i}^{(k)} (x)|\le C\varepsilon_{i}^{-k}\Big(\|\vec{u}\|+\varepsilon_{i}\|\vec{f}\|\Big)\;\; for\;\; k=1,2\;\\
 |u_{1}^{(3)} (x)|\le C\varepsilon_{1}^{-3}\Big(\|\vec{u}\|+\varepsilon_{1}\|\vec{f}\|\Big)+\varepsilon_{1}^{-1}\|f_{1}^{\prime}\|\hspace{0.55cm}\\
 |u_{2}^{(3)} (x)|\le C\varepsilon_{2}^{-2}\varepsilon_{1}^{-1}\Big(\|\vec{u}\|+\varepsilon_{2}\|\vec{f}\|\Big)+\varepsilon_{2}^{-1}\|f_{2}^{\prime}\|
\end{eqnarray}
\end{theorem}
\begin{proof}
\indent For any $x \in [0,1]$, there exists $a\in [0,1- \varepsilon_{i}$] such that $x \in N_{a}=[a,a+\varepsilon_{i}].$
By mean value theorem, there exists $y_{i} \in (a, a+\varepsilon_{i})$ such that
$$u_{i}^{\prime}(y_{i})=\frac{u_{i}(a+\varepsilon_{i})-u_{i}(a)}{\varepsilon_{i}}$$
and hence
$$|u_{i}^{\prime}(y_{i})|\le C\varepsilon_{i}^{-1}\|\vec u\|.$$
Also, $$u_{i}^{\prime}(x)=u_{i}^{\prime}(y_{i})+\int_{y_{i}}^{x}u_{i}^{\prime\prime}(s)ds$$
Substituting for $u_i^{\prime\prime}(s)$ from (\ref{e1}) and integrating by parts, we get 
$$|u_{i}^{\prime}(x)| \le C\varepsilon_{i}^{-1}\Big(\|\vec u\|+ \varepsilon_{i}\|\vec f\|\Big).$$
Again from (\ref{e1}),
$$|u_{i}^{\prime\prime}(x)|\le C\varepsilon_{i}^{-2}\Big(\|\vec u\|+ \varepsilon_{i}\|\vec f\|\Big).$$
Differentiating (\ref{e1}) once and substituting the above bounds lead to
$$|u_1^{(3)}(x)|\le C\varepsilon_{1}^{-3}\Big(\|\vec u\|+ \varepsilon_{1}\|\vec f\|\Big)+\varepsilon_{1}^{-1}\|f_1^\prime\|\;\;\quad$$
$$|u_2^{(3)}(x)|\le C\varepsilon_{2}^{-2}\varepsilon_{1}^{-1}\Big(\|\vec u\|+ \varepsilon_{2}\|\vec f\|\Big)+\varepsilon_{2}^{-1}\|f_2^\prime\|.$$
\end{proof}
\subsection{Shishkin decomposition of the solution}
The solution $\vec{u}$ of the problem (\ref{e1})-(\ref{e2}) can be decomposed into smooth and singular components $\vec{v}$ and $\vec{w}$ given by
$$\vec u=\vec v+\vec w$$
where 
\begin{eqnarray}\label{ev}
L\vec{v}=\vec{f}, \vec{v}(1)=\vec{r},\; \vec{v}(0)\;\text{suitably chosen},\\ \label{ew}
L\vec{w}=\vec{0}, \vec{w}(0)=\vec{l}-\vec{v}(0), \vec{w}(1)=\vec{0} \qquad\;
\end{eqnarray}
with $\vec{v}=(v_1, v_2)^T$ and $\vec{w}=(w_1, w_2)^T.$\\
\\
Now, $\vec{v}$ is decomposed into $\vec v=\vec y_0+\varepsilon_2 \vec y_1+\varepsilon_2^2 \vec y_2$, where\\
$\vec y_0=( y_{01}, y_{02})^T$ is the solution of (\ref{e8})-(\ref{e10}),
\begin{eqnarray}\label{e8}
a_1(x)y_{01}^{\prime}(x)-b_{11}(x)y_{01}(x)+b_{12}(x)y_{02}(x)=f_1(x)\\\label{e9}
a_2(x)y_{02}^{\prime}(x)+b_{21}(x)y_{01}(x)-b_{22}(x)y_{02}(x)=f_2(x)\\\label{e10}
y_{01}(1)=r_1(1),\;y_{02}(1)=r_2(1),\hspace{3.1cm}
\end{eqnarray}
$\vec y_1=(y_{11},y_{12})^T$ is the solution of (\ref{e11})-(\ref{e13}),
\begin{eqnarray}\label{e11}
a_1(x)y_{11}^{\prime}(x)-b_{11}(x)y_{11}(x)+b_{12}(x)y_{12}(x)=-\frac{\varepsilon_1}{\varepsilon_2}y_{01}^{\prime\prime}(x)\\\label{e12}
a_2(x)y_{12}^{\prime}(x)+b_{21}(x)y_{11}(x)-b_{22}(x)y_{12}(x)=-y_{02}^{\prime\prime}(x)\;\;\;\;\\\label{e13}
y_{11}(1)=0,\;y_{12}(1)=0\hspace{5.0cm}
\end{eqnarray}
$\vec y_2=(y_{21},y_{22})^T$ is the solution of (\ref{e14})-(\ref{e16}),
\begin{eqnarray}\label{e14}
\qquad \varepsilon_1y_{21}^{\prime\prime}(x)+a_1(x)y_{21}^{\prime}(x)-b_{11}(x)y_{21}(x)+b_{12}(x)y_{22}(x)=-\frac{\varepsilon_1}{\varepsilon_2}y_{11}^{\prime\prime}(x)\\\label{e15}
\qquad \varepsilon_2y_{22}^{\prime\prime}(x)+a_2(x)y_{22}^{\prime}(x)+b_{21}(x)y_{21}(x)-b_{22}(x)y_{22}(x)=-y_{12}^{\prime\prime}(x)\;\;\;\;\\\label{e16}
y_{21}(0)=p,\;y_{22}(0)=0,\;y_{21}(1)=0,\;y_{22}(1)=0.\hspace{2.1cm}
\end{eqnarray}
In (\ref{e16}), $p$ is a constant to be chosen such that $|p| \le C$.\\
\\From (\ref{e8})-(\ref{e13}), it is not hard to see that, for $0\le k \le 3,$
\begin{equation}\label{e17}
\|\vec{y_0}^{(k)}\| \le C,\;\;\|\vec {y_1}^{(k)}\| \le C.
\end{equation}
Now, consider the equations (\ref{e14})-(\ref{e16}) and using Lemma~\ref{L2}
\begin{equation}
\|\vec{y_2}\| \le C.
\end{equation}
Using the estimate (\ref{e5}) from Theorem~\ref{T1}, we get,
\begin{equation}\label{e20}
|y_{22}^{(k)} (x)|\le C\varepsilon_{2}^{-k}\;\; for\;\; k=1,2
\end{equation}
From (\ref{e14}),
\begin{equation}
\varepsilon_1y_{21}^{\prime\prime}+a_1(x)y_{21}^{\prime}(x)-b_{11}(x)y_{21}(x)=-\frac{\varepsilon_1}{\varepsilon_2}y_{11}^{\prime\prime}(x)-b_{12}(x)y_{22}(x).
\end{equation}
Decompose $y_{21}$ as $y_{21}(x)=z_0(x)+\varepsilon_1z_1(x)+\varepsilon_1^2z_2(x)$ with
\begin{eqnarray}\label{a1}
a_1(x)z_0^{\prime}(x)-b_{11}(x)z_0(x)=-\frac{\varepsilon_1}{\varepsilon_2}y_{11}^{\prime\prime}(x)-b_{12}(x)y_{22}(x), z_0(1)=0,\\\label{a2}
a_1(x)z_1^{\prime}(x)-b_{11}(x)z_1(x)=-z_0^{\prime\prime}(x), z_1(1)=0,\hspace{2.8cm}\\\label{a3}
\qquad \quad \varepsilon_1z_2^{\prime\prime}(x)+a_1(x)z_2^{\prime}(x)-b_{11}(x)z_2(x)=-z_1^{\prime\prime}(x), z_2(0)=0,z_2(1)=0.
\end{eqnarray}
Estimating $z_0\; \text{and}\; z_1$ from (\ref{a1}) \& (\ref{a2}) and using Chapter 8 of \cite{MOS96} for the problem (\ref{a3}), the following estimates hold for $0\le k\le3$,
$$|z_0^{(k)}|<C(1+\varepsilon_2^{(1-k)}),\; |z_1^{(k)}|<C(1+\varepsilon_2^{-2}\varepsilon_1^{2-k}),\;|z_2^{(k)}|<C(1+\varepsilon_2^{-2}\varepsilon_1^{-k})$$
Then $p=z_0(0)+\varepsilon_1 z_1(0)$ and for $k=0,1,2,$
\begin{equation}\label{e21}
|y_{21}^{(k)} (x)|\le C\varepsilon_{2}^{-2}\;,\;
|y_{21}^{(3)} (x)|\le C\varepsilon_{1}^{-1}\varepsilon_{2}^{-2}.
\end{equation}
Differentiating (\ref{e15}) once and using (\ref{e20}) and (\ref{e21}) 
\begin{equation}\label{e22}
|y_{22}^{(3)} (x)|\le C\varepsilon_{2}^{-3}.
\end{equation}
Hence, from (\ref{e17}) - (\ref{e20}) and (\ref{e21}) - (\ref{e22}),  the estimates of the components $v_1=y_{01}+\varepsilon_2y_{11}+\varepsilon_2^2y_{21}$ and $v_2=y_{02}+\varepsilon_2y_{12}+\varepsilon_2^2y_{22}$ of $\vec{v}$ are as follows.
\begin{eqnarray}\label{e27}
|v_1^{(k)}(x)|\le C, \;\; |v_2^{(k)}(x)|\le C\; \text{for}\; 0\le k \le2,\\ \label{e28}
|v_1^{(3)}(x)|\le C\varepsilon_1^{-1}, \;\; |v_2^{(3)}(x)|\le C\varepsilon_2^{-1}.\hspace{1.1cm}
\end{eqnarray}

\begin{theorem}\label{T2}
Let $ \vec{w}(x)$  be the solution of (\ref{ew}), then for x $\in \overline\Omega$, the following estiamates hold.
\begin{eqnarray}\label{e29}
 |w_{1}(x)|\le C \exp{\dfrac{-\alpha x}{\varepsilon_2}},\; |w_{2}(x)|\le C \exp{\dfrac{-\alpha x}{\varepsilon_2}},\hspace{2.3cm}\\\label{e30}
 |w_{1}^{(k)}(x)|\le C \Big(\varepsilon_1^{-k} \exp{\dfrac{-\alpha x}{\varepsilon_1}}+\varepsilon_2^{-k} \exp{\dfrac{-\alpha x}{\varepsilon_2}}\Big), \text{for}\; k=1,2,3,\\\label{e31}
 |w_{2}^{(k)}(x)|\le C\varepsilon_2^{-k} \exp{\dfrac{-\alpha x}{\varepsilon_2}},\;\text{for}\; k=1,2,\hspace{3.1cm}\\\label{e32}
 |w_{2}^{(3)}(x)|\le C\varepsilon_2^{-1} \Big(\varepsilon_1^{-1} \exp{\dfrac{-\alpha x}{\varepsilon_1}}+\varepsilon_2^{-2} \exp{\dfrac{-\alpha x}{\varepsilon_2}}\Big).\hspace{1.8cm}
\end{eqnarray}
\end{theorem}
\begin{proof}
Estimates (\ref{e29})-(\ref{e31}) follow from Lemma 4 of \cite{ZC04}.\\
From (\ref{ew}), we have\\
$$\varepsilon_2 w_2^{\prime \prime}(x) + a_2(x)w_2^\prime (x)+b_{21}(x)w_1(x)-b_{22}(x)w_2(x)=0.$$
Differentiating the above equation once,
$$| \varepsilon_2 w_2^{(3)}(x)| \le C\big(|w_2^{\prime\prime}(x)|+|w_1^\prime(x)|\big)$$
and hence,
$$|w_{2}^{(3)}(x)|\le C\varepsilon_2^{-1} \Big(\varepsilon_1^{-1} \exp{\dfrac{-\alpha x}{\varepsilon_1}}+\varepsilon_2^{-2} \exp{\dfrac{-\alpha x}{\varepsilon_2}}\Big).$$
\end{proof}
\subsection{Improved estimates for the bounds of the singular components}
 Let $B_1(x)$ and $B_2(x)$ be the layer functions defined on $[0,1]$ as follows
$$B_1(x)=\exp{\dfrac{-\alpha x}{\varepsilon_1}}, \;\; B_2(x)=\exp{\dfrac{-\alpha x}{\varepsilon_2}}.$$
Using the arguments similar to those used in Lemma 5 of \cite{PVM10}, it is not hard to see that there exists point $x_s \in (0,\frac{1}{2})$ such that
\begin{equation}\label{e33}
\frac{B_1(x_s)}{\varepsilon_1^s}=\frac{B_2(x_s)}{\varepsilon_2^s},\; s=1,2,3
\end{equation}
and
\begin{equation} \label{e34}
\frac{B_1(x)}{\varepsilon_1^s}>\frac{B_2(x)}{\varepsilon_2^s},\; \text{for}\; x\in[0,x_s),\;\;
\frac{B_1(x)}{\varepsilon_1^s}<\frac{B_2(x)}{\varepsilon_2^s},\; \text{for}\; x\in(x_s,1].
\end{equation}
Now the singular components $w_1(x)$ and $w_2(x)$ are decomposed as follows
\begin{equation}\label{ns}
w_1(x)=w_{11}(x)+w_{12}(x),\;\;\;w_2(x)=w_{21}(x)+w_{22}(x),
\end{equation}
where, $w_{11}, w_{12}, w_{21}$ and $w_{22}$ are defined by
\begin{equation}\label{e38}
\begin{split}
w_{11}(x) =
\begin{cases}
\begin{aligned}[b]
\displaystyle \sum_{k=0}^{3}\big((x-x_3)^k/k!\big)w_1^{(k)}(x_3),\; \text{for}\; x\in [0,x_3)
\end{aligned}
\\[1ex]
\begin{aligned}[b]
w_1(x),\qquad\qquad\qquad\qquad\quad \; \text{for}\; x\in [x_3,1]
\end{aligned}
\end{cases}
\end{split}
\end{equation}
\begin{equation}\label{e39}
w_{12}(x)=w_1(x)-w_{11}(x) \hspace{4.7cm}
\end{equation}
\begin{equation}\label{e40}
\begin{split}
w_{21}(x) =
\begin{cases}
\begin{aligned}[b]
\displaystyle \sum_{k=0}^{3}\big((x-x_1)^k/k!\big)w_2^{(k)}(x_1),\; \text{for}\; x\in [0,x_1)
\end{aligned}
\\[1ex]
\begin{aligned}[b]
w_2(x),\qquad\qquad\qquad\qquad\quad \; \text{for}\; x\in [x_1,1]
\end{aligned}
\end{cases}
\end{split}\;\;\;
\end{equation}
\begin{equation}\label{e41}
w_{22}(x)=w_2(x)-w_{21}(x). \hspace{4.7cm}
\end{equation}
\begin{lemma}\label{L3}
Let $w_{11}, w_{12}, w_{21}$ and $w_{22}$ are as defined in (\ref{e38})-(\ref{e41}), then for $x\in\overline\Omega$, the following estimates hold.
\begin{eqnarray}
|w_{11}^{(3)}(x)|\le C\varepsilon_2^{-3}B_2(x),\;\;\;|w_{12}^{\prime\prime}(x)|\le C\varepsilon_1^{-2}B_1(x), \\
|w_{21}^{(3)}(x)|\le C\varepsilon_2^{-3}B_2(x),\;\;\;|w_{22}^{\prime\prime}(x)|\le C\varepsilon_2^{-2}B_1(x).
\end{eqnarray}
\end{lemma}
\begin{proof}
For $x\in[0,x_3)$, by the definition of $w_{11}(x)$ and using (\ref{e30}) and (\ref{e33}),
$$|w_{11}^{(3)}(x)|=|w_1^{(3)}(x_3)|\le C\varepsilon_2^{-3}B_2(x_3)\le C\varepsilon_2^{-3}B_2(x).$$
For $x\in[x_3,1]$, by the definition of $w_{11}(x)$ and using (\ref{e30}) and (\ref{e34}),
$$|w_{11}^{(3)}(x)|=|w_1^{(3)}(x)|\le C\varepsilon_2^{-3}B_2(x).$$
Hence, 
\begin{equation}\label{e44}
|w_{11}^{(3)}(x)|\le C\varepsilon_2^{-3}B_2(x), \text{on}\; \overline{\Omega}.
\end{equation}
Similar arguments lead to,
\begin{equation}\label{e45}
|w_{21}^{(3)}(x)|\le C\varepsilon_2^{-3}B_2(x), \text{on}\; \overline{\Omega}.
\end{equation}
Using  (\ref{e39}), (\ref{e30}), (\ref{e44}) and  (\ref{e34}), it is not hard to see that, for $x\in[0,x_3)$,
$$|w_{12}^{(3)}(x)|\le|w_1^{(3)}(x)|+|w_{11}^{(3)}(x)|\le C\varepsilon_1^{-3}B_1(x).$$
Since $w_{12}^{\prime\prime}(1)=0$, it follows that for any $x\in [0,1]$,
$$|w_{12}^{\prime\prime}(x)|=\Big|\int_x^1w_{12}^{(3)}(t)dt\Big| \le C\int_x^1\varepsilon_1^{-3}B_1(t)dt\le C\varepsilon_1^{-2}B_1(x).$$
Hence, 
\begin{equation}\label{e44}
|w_{12}^{\prime\prime}(x)|\le  C\varepsilon_1^{-2}B_1(x), \text{on}\; \overline{\Omega}.
\end{equation}
Similar arguments lead to,
\begin{equation}\label{e45}
|w_{22}^{\prime\prime}(x)|\le C\varepsilon_2^{-2}B_1(x), \text{on}\; \overline{\Omega}.
\end{equation}
\end{proof}
Now consider the alternate decomposition of the singular component $w_1(x)$ as below.
\begin{equation}\label{ns1}
w_1(x)=w_{11}(x)+w_{12}(x),
\end{equation}
where $w_{11}$ and $w_{12}$ are defined by
\begin{equation}\label{e38a}
\begin{split}
w_{11}(x) =
\begin{cases}
\begin{aligned}[b]
\displaystyle \sum_{k=0}^{2}\big((x-x_2)^k/k!\big)w_1^{(k)}(x_2),\; \text{for}\; x\in [0,x_2)
\end{aligned}
\\[1ex]
\begin{aligned}[b]
w_1(x),\qquad\qquad\qquad\qquad\quad \; \text{for}\; x\in [x_2,1]
\end{aligned}
\end{cases}
\end{split}
\end{equation}
\begin{equation}\label{e39a}
w_{12}(x)=w_1(x)-w_{11}(x). \hspace{4.5cm}
\end{equation}
Then, arguments similar to Lemma~\ref{L3} lead to
\begin{equation}
|w_{11}^{\prime\prime}(x)|\le C\varepsilon_2^{-2}B_2(x),\;\;\;|w_{12}^{\prime}(x)|\le C\varepsilon_1^{-1}B_1(x).
\end{equation}
\section{Numerical Method}
\label{sec:4}
A piecewise uniform Shishkin mesh $\overline{\Omega}^N$ is defined on $[0,1]$, so as to resolve the layers in the neighbourhood of $x=0$. Let N denote the number of mesh elements which is taken to be a multiple of 4. The interval $[0,1]$ is divided into three subintervals $ [0,\tau_1],\;[\tau_1, \tau_2]\; \text{and}\; [\tau_2,1]$, where $\tau_1$ and $\tau_2$ are the transition parameters given by,
\begin{equation*}
\tau_2=\text{min}\Big\{\frac{1}{2},\frac{2\varepsilon_2}{\alpha} \ln N\Big\},\;\;\;\;\tau_1=\text{min}\Big\{\frac{\tau_2}{2},\frac{2\varepsilon_1}{\alpha} \ln N\Big\}.
\end{equation*}

 In each of the intervals $ [0,\tau_1], [\tau_1, \tau_2]$, $N/4$ mesh elements are placed and $N/2$ mesh elements are placed in the interval $[\tau_2, 1]$ so that the mesh is piecewise uniform. The mesh becomes uniform when  $\tau_2=1/2$ and  $\tau_1=\tau_2/2.$  \\

Let $H_1, H_2$ and $H_3$ denote the step sizes in the intervals $ [0,\tau_1],\;[\tau_1, \tau_2]\; \text{and}\; [\tau_2,1]$ respectively. Thus,
$$H_1=\dfrac{4\tau_1}{N},\;\; H_2=\dfrac{4(\tau_2-\tau_1)}{N}\;\;\text{and}\; H_3=\dfrac{2(1-\tau_2)}{N}.$$
Therefore the possible four Shishkin meshes are represented by $\overline{\Omega}^N=\{x_j\}^N_{j=0}, $ where,
\begin{equation*}
\begin{split}
x_j=
\begin{cases}
\begin{aligned}[b]
jH_1,\hspace{2.3cm}\text{if}\;\; 0\le j\le \frac{N}{4}
\end{aligned}
\\[2.0ex]
\begin{aligned}[b]
\tau_1 +(j-\frac{N}{4})H_2 ,\hspace{0.45cm}\text{if} \;\frac{N}{4} \le j \le \frac{N}{2}
\end{aligned}
\\[2.0ex]
\begin{aligned}[b]
\tau_2 +(j-\frac{N}{2})H_3 ,\;\quad\text{if}\; \frac{N}{2} \le j \le N.
\end{aligned}
\end{cases}
\end{split}
\end{equation*}
To resolve the layers, the mesh is constructed in such a way that it condenses at the inner regions where the layers are exhibited and is coarse in the outer region, away from the layers.

To solve the BVP  (\ref{e1})-(\ref{e2}) numerically the following upwind classical finite difference scheme is applied on the mesh $\overline\Omega^N$. 
\begin{eqnarray}\label{e52}
L^N\vec{U}(x_j)\equiv E\delta^2\vec{U}(x_j)+A(x_j)D^+\vec{U}(x_j)-B(x_j)\vec{U}(x_j)=\vec{f}(x_j),\\\label{e53}
\vec{U}(x_0)=\vec{l},\;\vec{U}(x_N)=\vec{r},\hspace{6.3cm}
\end{eqnarray}
where, $\vec{U}(x_j)=(U_1(x_j),U_2(x_j))^T$ and for $1\le j\le N-1,$\\
$$ D^+U(x_j)=\frac{U(x_{j+1})-U(x_j)}{h_{j+1}},\;\; D^-U(x_j)=\frac{U(x_j)-U(x_{j-1})}{h_j},$$
$$\delta^2U(x_j)=\dfrac{1}{\overline h_j}\Big(D^+U(x_j)-D^-U(x_j)\Big),$$
with
$$h_j=x_j-x_{j-1},\;\; \overline{h_j}=\frac{(h_j+h_{j+1})}{2} .$$

\section{Error Analysis}
In this section a discrete maximum principle, a discrete stability result and the first order convergence of the proposed numerical method are established.
\begin{lemma}\label{L4}(Discrete Maximum Principle)
Assume that the vector valued mesh function $\vec{\psi}(x_j)=(\psi_1(x_j),\psi_2(x_j))^T$ satisfies $\vec{\psi}(x_0) \ge \vec{0}$ and $\vec{\psi}(x_N)\ge \vec{0}$. Then $L^N\vec{\psi}(x_j) \le \vec0$ for $1\le j\le N-1$ implies that $\vec{\psi}(x_j) \ge \vec0$ for $0 \le j \le N.$
\end{lemma}
\begin{proof}
Let $k_1$ and $k_2$ be such that $\psi_1(x_{k_1})=\displaystyle{\min_j\psi_1(x_j)}$ and $\psi_2(x_{k_2})=\displaystyle{\min_j\psi_2(x_j)}$. Without loss of generality, we assume that $\psi_1(x_{k_1}) \le \psi_2(x_{k_2})$ and suppose $\psi_1(x_{k_1}) < 0$.
Then, $k_1 \not\in\{0,N\},\; \psi_1(x_{k_1+1})-\psi_1(x_{k_1}) \ge 0$ and $\psi_1(x_{k_1}) - \psi(x_{k-1}) \le 0$, implies that $(L^N\vec{\psi})_1(x_{k_1}) > 0$, a contradiction.
Therefore, $\psi_1(x_{k_1}) \ge 0$ and hence, $\vec{\psi}(x_j) \ge \vec{0}$ for $0\le j \le N$.
\end{proof}
An immediate consequence of the above discrete maximum principle is the following discrete stability result.
\begin{lemma}\label{L5}(Discrete Stability Result)
If $\vec{\psi}(x_j)=(\psi_1(x_j),\psi_2(x_j))^T$ is any vector valued mesh function defined on $\overline\Omega^N$, then for $i=1,2$ and $0\le j \le N$,\\
$$|\psi_i(x_j)| \le max\Big\{\|\vec{\psi}(x_0)\|,\;\|\vec{\psi}(x_N)\|,\; \frac{1}{\beta}\;\|L^N\vec{\psi}\|_{\Omega^N}\Big\}.$$
\end{lemma}
\subsection{Error Estimate}
\label{sec:8}
\indent Analogous to the continuous case, the discrete solution $\vec{U}$ can be decomposed into $\vec{V}$ and $\vec{W}$ as defined below.
\begin{equation}\label{e50}
L^N\vec{V}(x_j) = \vec{f}(x_j), \;\text{for} \; 0 < j < N,\; \vec{V}(x_0)=\vec{v}(x_0),\; \vec{V}(x_N)=\vec{v}(x_N)\\
\end{equation}
\begin{equation}\label{e51}
L^N\vec{W}(x_j) = \vec{0}, \;\text{for} \; 0 < j < N,\; \vec{W}(x_0)=\vec{w}(x_0),\; \vec{W}(x_N)=\vec{w}(x_N)\;\;\;\;\\
\end{equation}
\begin{lemma}\label{L6}
Let $\vec{v}$ be the solution of (\ref{ev}) and $\vec{V}$ be the solution of (\ref{e50}), then
$$\|\vec{V}-\vec{v}\|_{\overline\Omega^N} \le CN^{-1}.$$
\end{lemma}
\begin{proof}
For $ 1\le j\le N-1$,
\begin{align*}
L^N(\vec{V}-\vec{v})(x_j)&=\vec{f}(x_j)-L^N\vec{v}(x_j)\\
&=(L-L^N)\vec{v}(x_j)\\
&=(\frac{d^2}{dx^2}-\delta^2)E\vec{v}(x_j)+(\frac{d}{dx}-D^+)A(x_j)\vec{v}(x_j)\\
&=
\begin{pmatrix}
\varepsilon_1(\frac{d^2}{dx^2}-\delta^2)v_1(x_j)+a_1(x_j)(\frac{d}{dx}-D^+)v_1(x_j)\\
\varepsilon_2(\frac{d^2}{dx^2}-\delta^2)v_2(x_j)+a_2(x_j)(\frac{d}{dx}-D^+)v_2(x_j)
\end{pmatrix}.
\end{align*}
By the standard local truncation used in the taylor expansions,
$$|\varepsilon_1(\frac{d^2}{dx^2}-\delta^2)v_1(x_j)+a_1(x_j)(\frac{d}{dx}-D^+)v_1(x_j)|\le C(x_{j+1}-x_{j-1})(\varepsilon_1\|v_1^{(3)}\|+\|v_1^{(2)}\|),$$
$$|\varepsilon_2(\frac{d^2}{dx^2}-\delta^2)v_2(x_j)+a_2(x_j)(\frac{d}{dx}-D^+)v_2(x_j)|\le C(x_{j+1}-x_{j-1})(\varepsilon_2\|v_2^{(3)}\|+\|v_2^{(2)}\|).$$
Since $(x_{j+1}-x_{j-1}) \le CN^{-1}$, using (\ref{e27}) and (\ref{e28}),
$$\|L^N(\vec{V}-\vec{v})\|_{\Omega^N} \le CN^{-1}.$$
Using Lemma~\ref{L5},
\begin{equation}\label{eev}
\|\vec{V}-\vec{v}\|_{\overline\Omega^N} \le CN^{-1}.
\end{equation}
\end{proof}
To estimate the error in the singular components , we consider the mesh functions $B_1^N(x_j)$ and $B_2^N(x_j)$  on $\overline\Omega^N$ defined by
$$B_1^N(x_j)=\displaystyle\prod_{i=1}^j(1+\frac{\alpha h_i}{2\varepsilon_1})^{-1} \;\text{and}\;B_2^N(x_j)=\displaystyle\prod_{i=1}^j(1+\frac{\alpha h_i}{2\varepsilon_2})^{-1}$$
with $B_1^N(x_0)=B_2^N(x_0)=1.$\\
It is to be observed that $B_1^N$ and $B_2^N$ are monotonically decreasing.
\begin{lemma}\label{L7}
The layer components $W_1$ and $W_2$ satisfy the following bounds on $\overline\Omega^N$.
$$|W_1(x_j)| \le C B^N_2(x_j)\; \text{and}\;\;|W_2(x_j)| \le C B^N_2(x_j).$$ 
\end{lemma}
\begin{proof}
Consider the following vector valued mesh functions  on $\overline\Omega^N$, 
$$\vec{\psi}^{\pm}(x_j)=C\big(B^N_2(x_j), B^N_2(x_j)\big)^T \pm \vec{W}(x_j).$$
Then for sufficiently large C, $\vec\psi^{\pm}(x_0) \ge \vec{0}$, $ \vec\psi^{\pm}(x_N) \ge\vec{0}$ and\\
$$L^N \vec{\psi}^{\pm}(x_j) = CL^N
\begin{pmatrix}
B^N_2(x_j)\\
B^N_2(x_j)
\end{pmatrix} 
\le \vec{0}.$$
Using discrete maximum principle, we have $\vec{\psi}^{\pm}(x_j)\ge\vec{0}$ on $\overline\Omega^N,$ which implies that
$$|W_1(x_j)| \le C B^N_2(x_j)\; \text{and}\;|W_2(x_j)| \le C B^N_2(x_j).$$ 
\end{proof}

\begin{lemma}\label{L8}
Let $\vec{w}$ be the solution of (\ref{ew}) and $\vec{W}$ be the solution of (\ref{e51}), then
$$\|\vec{W}-\vec{w}\|_{\overline\Omega^N}  \le CN^{-1}\ln N.$$
\end{lemma}
\begin{proof}
By the standard local truncation used in the Taylor expansions,
$$|\varepsilon_1(\frac{d^2}{dx^2}-\delta^2)w_1(x_j)+a_1(x_j)(\frac{d}{dx}-D^+)w_1(x_j)|\le C(x_{j+1}-x_{j-1})(\varepsilon_1\|w_1^{(3)}\|+\|w_1^{(2)}\|)$$
$$|\varepsilon_2(\frac{d^2}{dx^2}-\delta^2)w_2(x_j)+a_2(x_j)(\frac{d}{dx}-D^+)w_2(x_j)|\le C(x_{j+1}-x_{j-1})(\varepsilon_2\|w_2^{(3)}\|+\|w_2^{(2)}\|)$$
where the norm is taken over the interval ${[x_{j-1},x_{j+1} ]}$.\\

For the case $\tau_2=1/2$ and $\tau_1=1/4$, the mesh is uniform, $h=N^{-1}$,  $\varepsilon_1^{-1}\le C \ln N$ and $\varepsilon_2^{-1}\le C \ln N$ and thus we obtain,
\begin{equation}\label{w64}
|L^N(\vec{W}-\vec{w})(x_j)|\le
\begin{pmatrix}
CN^{-1}\big(\varepsilon_1^{-2}B_1(x_{j-1})+\varepsilon_2^{-2}B_2(x_{j-1})\big)\;\;\;\;\;\\
CN^{-1}\big(\varepsilon_1^{-1}\varepsilon_2^{-1}B_1(x_{j-1})+\varepsilon_2^{-2}B_2(x_{j-1})\big)
\end{pmatrix}.
\end{equation}
Consider the following barrier function $\vec\phi$ given by
$$\phi_1(x_j)=\frac{CN^{-1}}{\gamma(\alpha-\gamma)}\Big(\exp(2\gamma h/\varepsilon_1)\varepsilon_1^{-1}Y_j+\exp(2\gamma h/\varepsilon_2)\varepsilon_2^{-1}Z_j\Big)$$
$$\phi_2(x_j)=\frac{CN^{-1}}{\gamma(\alpha-\gamma)}\Big(\exp(2\gamma h/\varepsilon_2)\varepsilon_1^{-1}Z_j\Big)\hspace{3.3cm}$$
where $\gamma$ is a constant such that $0<\gamma< \alpha$,
$$Y_j=\dfrac{\lambda^{N-j}-1}{\lambda^N-1}\; \text{with}\; \lambda=1+\frac{\gamma h}{\varepsilon_1} $$
and
$$Z_j=\dfrac{\Lambda^{N-j}-1}{\Lambda^N-1}\; \text{with}\; \Lambda=1+\frac{\gamma h}{\varepsilon_2}.$$
It is not hard to see that
$$0\le Y_j, Z_j \le 1,$$
$$ (\varepsilon_1\delta^2+\gamma D^+)Y_j=0,\;\; (\varepsilon_2\delta^2+\gamma D^+)Z_j=0, $$
$$ D^+Y_j\le-\frac{\gamma}{\varepsilon_1}\exp(-\gamma x_{j+1}/\varepsilon_1),\;\; D^+Z_j\le-\frac{\gamma}{\varepsilon_2}\exp(-\gamma x_{j+1}/\varepsilon_2).$$
Hence, 
\begin{align}\nonumber
(L^N\vec\phi)(x_j)&\le\dfrac{CN^{-1}}{\gamma(\alpha-\gamma)}
\begin{pmatrix}
\varepsilon_1^{-1}\exp(2\gamma h/\varepsilon_1)D^+Y_j+\varepsilon_2^{-1}\exp(2\gamma h/\varepsilon_2)D^+Z_j\\
\varepsilon_1^{-1}\exp(2\gamma h/\varepsilon_2)(a_2-\gamma)D^+Z_j
\end{pmatrix}\\\label{w67}
&\le -CN^{-1}
\begin{pmatrix}
\varepsilon_1^{-2}B_1(x_{j-1})+\varepsilon_2^{-2}B_2(x_{j-1})\\
\varepsilon_1^{-1}\varepsilon_2^{-1}B_1(x_{j-1})
\end{pmatrix}
\end{align}
Consider the discrete functions
$$\vec{\psi}^{\pm}(x_j)=\vec{\phi}(x_j)\pm (\vec{W}-\vec{w})(x_j), x_j \in \overline\Omega^N.$$
Then for sufficiently large C, using (\ref{w64}) and (\ref{w67}),  $\vec{\psi}^{\pm}(x_0) > \vec{0}$, $\vec{\psi}^{\pm}(x_N)=\vec{0}$ and $L^N\vec{\psi}^{\pm}(x_j)\le\vec{0}$ on $\Omega^N$.\\
Using discrete maximum principle, $\vec{\psi}^{\pm}(x_j)\ge\vec{0}$ on $\overline\Omega^N$. Hence,\\
$$|(\vec{W}-\vec{w})(x_j)|\le
\begin{pmatrix}
CN^{-1}\big(\varepsilon_1^{-1}+\varepsilon_2^{-1}\big)\\
CN^{-1}\big(\varepsilon_1^{-1}\big)
\end{pmatrix}
\le
\begin{pmatrix}
CN^{-1}\ln N\\
CN^{-1}\ln N
\end{pmatrix}
$$
implies that
\begin{equation}\label{e57}
\|(\vec{W}-\vec{w})\|_{\overline\Omega^N}  \le CN^{-1}\ln N.
\end{equation}
For other choices of $\tau_1$ and $\tau_2$, estimate of $\|(\vec{W}-\vec{w})\|_{\overline\Omega^N} $ is as follows.\\
\\
Let $\overline{\Omega}_1^N=\big\{x_j\big\}^{N/2}_{j=0}$ and $\overline{\Omega}_2^N=\big\{x_j\big\}^{N}_{j=N/2}$, then for $x_j \in \overline{\Omega}_2^N$, using Lemma~\ref{L7} and Theorem~\ref{T2},
\begin{align*}
|(W_1-w_1)(x_j)|&\le|W_1(x_j)|+|w_1(x_j)|\le CB_2^N(x_j)+CB_2(x_j)\\
&\le CB_2^N(\tau_2)+CB_2(\tau_2).
\end{align*}
$$B_2(\tau_2)=\exp{(\dfrac{-\alpha \tau_2}{\varepsilon_2})}\le\exp{(-\ln N)}\le N^{-1}.$$
\begin{align*}
B_2^N(\tau_2)&=\displaystyle\prod_{i=1}^j\big(1+\frac{\alpha h_i}{2\varepsilon_2}\big)^{-1}\\
&=\Big(1+\frac{\alpha H_1}{2\varepsilon_2}\Big)^{\frac{-N}{4}}\Big(1+\frac{\alpha H_2}{2\varepsilon_2}\Big)^{\frac{-N}{4}}\\
&=\Big(1+\frac{2\alpha\tau_1}{N\varepsilon_2}\Big)^{\frac{-N}{4}}\Big(1+\frac{2\alpha(\tau_2-\tau_1)}{N\varepsilon_2}\Big)^{\frac{-N}{4}}\\
&\le \Big(1+\frac{2\alpha\tau_2}{N\varepsilon_2}\Big)^{\frac{-N}{4}}\\
B_2^N(\tau_2)&\le N^{-1}.
\end{align*}
Hence, $|(W_1-w_1)(x_j)| \le CN^{-1}.$\\
Similarly, it is true that $|(W_2-w_2)(x_j)| \le CN^{-1}$ and hence,
\begin{equation}\label{e58}
\|(\vec{W}-\vec{w})\|_{\overline{\Omega}_2^N} \le CN^{-1}.
\end{equation}
For $N/4 \le j < N/2$, if $\varepsilon_2/2\le\varepsilon_1 \le \varepsilon_2 $, then $\tau_2\le (4\varepsilon_1/\alpha) \ln N$ implies that
\begin{equation}
|L^N(\vec{W}-\vec{w})(x_j)|\le CN^{-1}\ln N
\begin{pmatrix}
\varepsilon_1^{-1}B_1(x_{j-1})+\varepsilon_2^{-1}B_2(x_{j-1})\\
\varepsilon_2^{-1}B_2(x_{j-1})
\end{pmatrix}.
\end{equation}
On the other hand, if $\varepsilon_2 > 2 \varepsilon_1 $, then using (\ref{ns}) ,
\begin{equation*}
\begin{pmatrix}
|\varepsilon_1(\frac{d^2}{dx^2}-\delta^2)w_1(x_j)|\\
|\varepsilon_2(\frac{d^2}{dx^2}-\delta^2)w_2(x_j)|
\end{pmatrix}
\le
\begin{pmatrix}
|\varepsilon_1(\frac{d^2}{dx^2}-\delta^2)w_{11}(x_j)|\\
|\varepsilon_2(\frac{d^2}{dx^2}-\delta^2)w_{21}(x_j)|
\end{pmatrix}
+
\begin{pmatrix}
|\varepsilon_1(\frac{d^2}{dx^2}-\delta^2)w_{12}(x_j)|\\
|\varepsilon_2(\frac{d^2}{dx^2}-\delta^2)w_{22}(x_j)|
\end{pmatrix}.
\end{equation*}
Also, by the standard local truncation used in the Taylor expansions and using Lemma~\ref{L3},
\begin{align*}
\begin{pmatrix}
|\varepsilon_1(\frac{d^2}{dx^2}-\delta^2)w_{11}(x_j)|\\
|\varepsilon_2(\frac{d^2}{dx^2}-\delta^2)w_{21}(x_j)|
\end{pmatrix}
&\le
\begin{pmatrix}
C\varepsilon_1(x_{j+1}-x_{j-1})\|w_{11}^{(3)}\|\\
C\varepsilon_2(x_{j+1}-x_{j-1})\|w_{21}^{(3)}\|
\end{pmatrix}\\
&\le C\varepsilon_2^{-1}N^{-1}\ln N
\begin{pmatrix}
B_2(x_{j-1})\\
B_2(x_{j-1})
\end{pmatrix},\\
\begin{pmatrix}
|\varepsilon_1(\frac{d^2}{dx^2}-\delta^2)w_{12}(x_j)|\\
|\varepsilon_2(\frac{d^2}{dx^2}-\delta^2)w_{22}(x_j)|
\end{pmatrix}
&\le C
\begin{pmatrix}
\varepsilon_1\|w_{12}^{\prime\prime}\|_{[x_{j-1}, x_{j+1}]}\\
\varepsilon_2\|w_{22}^{\prime\prime}\|_{[x_{j-1}, x_{j+1}]}
\end{pmatrix}\\
&\le C
\begin{pmatrix}
\varepsilon_1^{-1}B_1(x_{j-1})\\
\varepsilon_2^{-1}B_1(x_{j-1})
\end{pmatrix}
\end{align*}
Thus, for $N/4 \le j < N/2$,\\
\begin{equation}\label{d1}
\begin{pmatrix}
|\varepsilon_1(\frac{d^2}{dx^2}-\delta^2)w_1(x_j)|\\
|\varepsilon_2(\frac{d^2}{dx^2}-\delta^2)w_2(x_j)|
\end{pmatrix}
\le
\begin{pmatrix}
C\varepsilon_2^{-1}N^{-1}\ln NB_2(x_{j-1})+C\varepsilon_1^{-1}B_1(x_{j-1})\\
C\varepsilon_2^{-1}N^{-1}\ln NB_2(x_{j-1})+C\varepsilon_2^{-1}B_1(x_{j-1})
\end{pmatrix}.
\end{equation}
Using the alternate decomposition of $w_1(x)$ given in (\ref{ns1}) and the arguments similar to the above, it is not hard to verify that for $N/4 \le j < N/2$,
\begin{equation}\label{d2}
\begin{pmatrix}
|(\frac{d}{dx}-D^+)w_1(x_j)|\\
|(\frac{d}{dx}-D^+)w_2(x_j)|
\end{pmatrix}
\le
\begin{pmatrix}
C\varepsilon_2^{-1}N^{-1}\ln NB_2(x_{j-1})+C\varepsilon_1^{-1}B_1(x_{j-1})\\
C\varepsilon_2^{-1}N^{-1}\ln NB_2(x_{j-1})+C\varepsilon_2^{-1}B_1(x_{j-1})
\end{pmatrix}.
\end{equation}
Hence, for $N/4 \le j < N/2$, expressions (\ref{d1}) \& (\ref{d2}) yield\\
\begin{equation}\label{e64}
|L^N(\vec{W}-\vec{w})(x_j)|\le
\begin{pmatrix}
C\varepsilon_2^{-1}N^{-1}\ln NB_2(x_{j-1})+C\varepsilon_1^{-1}B_1(x_{j-1})\\
C\varepsilon_2^{-1}N^{-1}\ln NB_2(x_{j-1})+C\varepsilon_2^{-1}B_1(x_{j-1})
\end{pmatrix}.
\end{equation}
For $0 < j < N/4$, $\tau_1\le (\varepsilon_1/\alpha)\ln N$ and hence
\begin{equation}\label{e65}
|L^N(\vec{W}-\vec{w})(x_j)|\le CN^{-1}\ln N
\begin{pmatrix}
\varepsilon_1^{-1}B_1(x_{j-1})+\varepsilon_2^{-1}B_2(x_{j-1})\\
\varepsilon_2^{-1}B_2(x_{j-1})
\end{pmatrix}.
\end{equation}
Consider the following barrier functions for $0<j<N/4$
\begin{eqnarray}
\;\qquad\phi_1(x_j)=CN^{-1}\ln N \big(\exp (2\alpha H_1/\varepsilon_1)B_1^N(x_j)+\exp (2\alpha H_1/\varepsilon_2)B_2^N(x_j)\big)\\
\phi_2(x_j)=CN^{-1}\ln N \exp (2\alpha H_1/\varepsilon_2)B_2^N(x_j)\hspace{4.0cm}
\end{eqnarray}
and for  $N/4\le j\le N/2$,
\begin{eqnarray}
\phi_1(x_j)=CN^{-1}\ln N \exp (2\alpha H_2/\varepsilon_2)B_2^N(x_j)+C B_1^N(x_j)\hspace{2.3cm}\\
\qquad \phi_2(x_j)=CN^{-1}\ln N \exp (2\alpha H_2/\varepsilon_2)B_2^N(x_j)+CN^{-1}\big((\tau_2-x_j)\varepsilon_2^{-1}+1\big)
\end{eqnarray}
Let $\vec\phi=(\phi_1,\phi_2)^T$ and consider the following vector valued mesh functions, for \;$0 \le j \le N/2$, 
$$\vec{\psi}^{\pm}(x_j)=\vec\phi(x_j) \pm (\vec{W}-\vec{w})(x_j).$$
For sufficiently large C,
$$\vec{\psi}^{\pm}(x_0)\ge \vec{0},\; \vec{\psi}^{\pm}(x_{\frac{N}{2}})\ge \vec{0}\; \text{and}\; L^N\vec{\psi}^{\pm}(x_j)\le \vec{0}, \text{for}\; 0 < j < N/2.$$
Then by Lemma~\ref{L4} $\vec{\psi}^{\pm}(x_j)\ge\vec{0}$  for $0 \le j \le N/2.$
Hence,
\begin{equation}\label{e69}
\|(\vec{W}-\vec{w})\|_{\overline{\Omega}_1^N}  \le CN^{-1}\ln N.
\end{equation}
Therefore, for any choice of $\tau_1$ and $\tau_2$,
\begin{equation}\label{e70}
\|(\vec{W}-\vec{w})\|_{\overline\Omega^N}  \le CN^{-1}\ln N.
\end{equation}
\end{proof}
\begin{theorem}\label{T4}
Let $\vec{u}$ be the solution of the problem  (\ref{e1})-(\ref{e2}) and $\vec{U}$ be the solution of the problem (\ref{e52})-(\ref{e53}), then,
\begin{equation*}
\|(\vec{u}-\vec{U})\|_{\overline\Omega^N} \le 
CN^{-1} \ln N.
\end{equation*}
\end{theorem}
\begin{proof}
The result follows by using triangle inequality, (\ref{eev}) and (\ref{e70}).
\end{proof}
\section{Numerical Illustrations}
\label{sec:9}
\begin{example}\label{ex1}
Consider the boundary value problem for the system of convection diffusion equations on (0,1)
\begin{eqnarray}\label{ss1}
\varepsilon_1 u_1^{\prime \prime}(x)+(1+x^2)u_1^\prime (x)-(4+\sin x)u_1(x)+2u_2(x)=-e^x,\\
\label{ss2}
\varepsilon_2 u_2^{\prime \prime}(x)+(2+x)u_2^\prime (x)+u_1(x)-(2+\cos x)u_2(x)=-x^2,\\ \label{ss3}
 \text{with}\; u_1(0)  =3,\;u_2(0)=3,\; u_1(1)  =1,\;u_2(1)=1.\hspace{2.0cm}
\end{eqnarray}
\end{example}
The above problem is solved using the suggested numerical method and  plot of the approximate solution for $N=1024, \varepsilon_1=5^{-4}, \varepsilon_2=2^{-7}$ is shown in Figure~\ref{f1}. Parameter uniform error and order of convergence of the numerical method are shown in Table~\ref{tb1} which are computed using two mesh algorithm, a variant of the one suggested in \cite{FHM00}.\\
\begin{figure}[h]
\includegraphics[width=0.50\textwidth, angle=270]{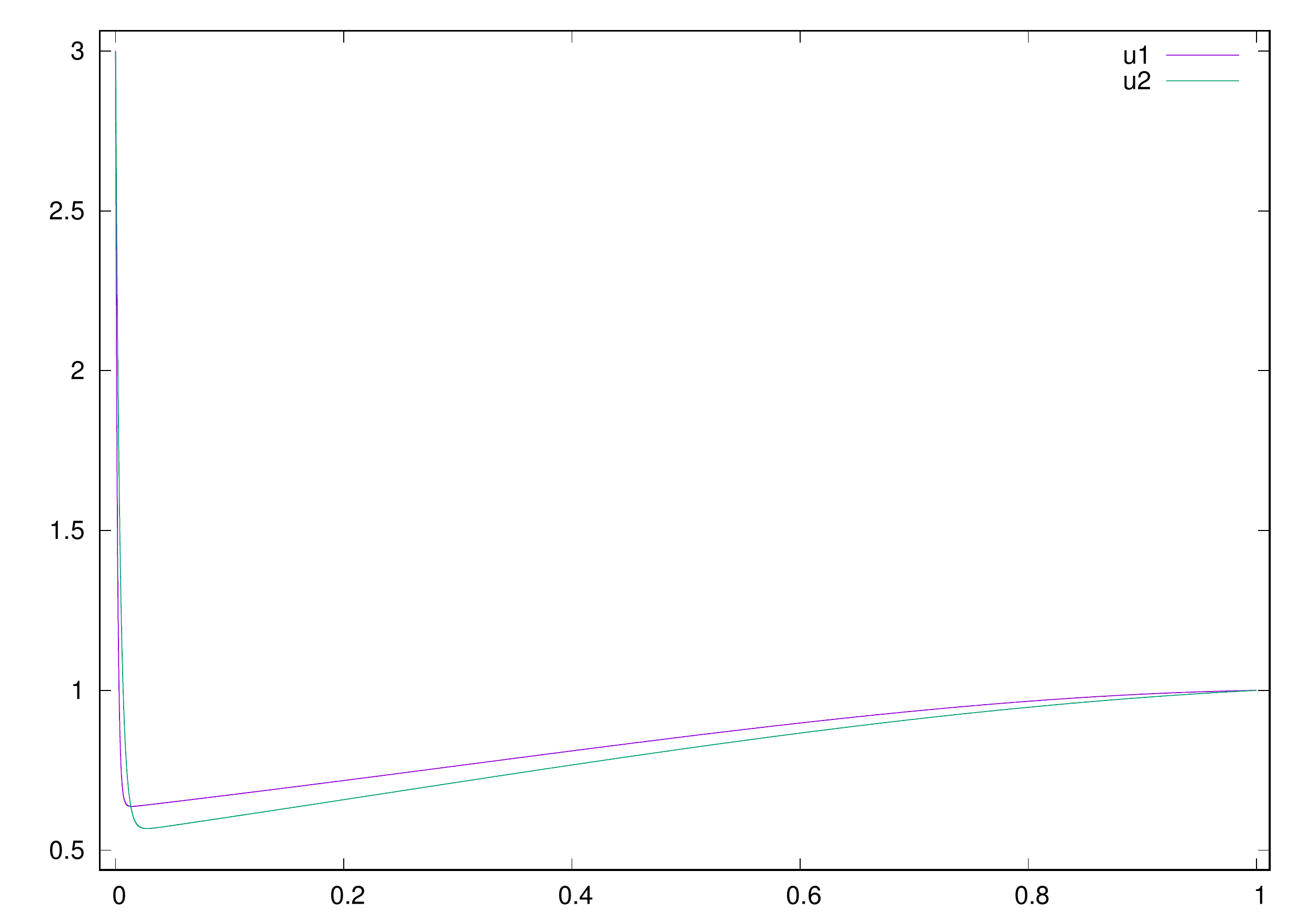}
\caption{Approximate solution of Example~\ref{ex1}.}
\label{f1}
\end{figure}
 
\begin{table}[h]\Small
\caption{}\label{tb1}
\renewcommand\arraystretch{1.1}
\noindent\[
\begin{array}{|c|c| c c c c c|}
\hline
 &&\multicolumn{5}{|c|}{\text{Number of mesh elements}\; N}\\
\cline{3-7}
\varepsilon_1&\varepsilon_2&128& 256&512&1024&2048\\
\hline
5^{-4}&2^{-7}&4.725E-02&2.887E-02&1.775E-02&1.019E-02&5.779E-03\\
5^{-5}&2^{-8}&4.789E-02&2.919E-02&1.792E-02&1.028E-02&5.827E-03\\
5^{-6}&2^{-9}&6.282E-02&4.456E-02&2.644E-02&1.535E-02&8.425E-03\\
5^{-7}&2^{-10}&7.146E-02&5.089E-02&3.212E-02&1.914E-02&1.095E-02\\
5^{-8}&2^{-11}&7.393E-02&5.243E-02&3.365E-02&2.009E-02&1.159E-02\\
5^{-9}&2^{-12}&7.470E-02&5.321E-02&3.437E-02&2.033E-02&1.174E-02\\
5^{-10}&2^{-13}&7.497E-02&5.355E-02&3.462E-02&2.040E-02&1.177E-02\\
5^{-11}&2^{-14}&7.508E-02&5.367E-02&3.471E-02&2.042E-02&1.179E-02\\
5^{-12}&2^{-15}&7.512E-02&5.372E-02&3.475E-02&2.042E-02&1.180E-02\\
5^{-13}&2^{-16}&7.513E-02&5.374E-02&3.477E-02&2.043E-02&1.181E-02\\
5^{-14}&2^{-17}&7.514E-02&5.375E-02&3.477E-02&2.043E-02&1.181E-02\\
5^{-15}&2^{-18}&7.514E-02&5.375E-02&3.478E-02&2.044E-02&1.181E-02\\
5^{-16}&2^{-19}&7.515E-02&5.376E-02&3.478E-02&2.044E-02&1.181E-02\\
5^{-17}&2^{-20}&7.515E-02&5.376E-02&3.478E-02&2.044E-02&1.181E-02\\
5^{-18}&2^{-21}&7.515E-02&5.376E-02&3.478E-02&2.044E-02&1.181E-02\\
\hline
\hline
D^N&&7.515E-02&5.376E-02&3.478E-02&2.044E-02&1.181E-02\\
p^N&&0.483E+00&0.628E+00&0.767E+00&0.791E+00&\\
C_p^N&&2.755E+00&2.755E+00&2.491E+00&2.047E+00&1.654E+00\\
\hline
\end{array}
\]
Computed order of $(\varepsilon_1,\varepsilon_2)$-uniform convergence, $p^* =  0.4833.$\\
Computed $(\varepsilon_1,\varepsilon_2)$-uniform error constant, $C_{p^*}^N=2.7546.\;\;\;\;$
\end{table}
 From Table~\ref{tb1}, it is to be noted that the error decreases as number of mesh elements N increases. Also for each N, the error stabilizes as $\varepsilon_1$ and $\varepsilon_2$ tends to zero.\\
\begin{example}\label{ex2}
Consider the boundary value problem for the system of convection diffusion equations on (0,1)
\begin{eqnarray}\label{ex21}
\varepsilon_1 u_1^{\prime \prime}(x)+u_1^\prime (x)-2u_1(x)+u_2(x)=-3(x-1),\hspace{1.5cm}\\
\label{ex22}
\varepsilon_2 u_2^{\prime \prime}(x)+(1+x)u_2^\prime (x)+xu_1(x)-(2x+1)u_2(x)=-2x,\\ \label{ex23}
 \text{with}\; u_1(0)  =0,\;u_2(0)=3,\; u_1(1)  =2,\;u_2(1)=2.\hspace{1.3cm}
\end{eqnarray}
\end{example}
The reduced problem corresponding to (\ref{ex21}) - (\ref{ex23}) is
\begin{eqnarray}\label{r1}
u_{01}^\prime (x)-2u_{01}(x)+u_{02}(x)=-3(x-1),\hspace{1.4cm}\\
\label{r2}
(1+x)u_{02}^\prime (x)+xu_{01}(x)-(2x+1)u_{02}(x)=-2x,\\ \label{r3}
 \text{with}\; u_{01}(1)  =2,\;u_{02}(1)=2.\hspace{3.1cm}
\end{eqnarray}
\indent Solution of the reduced problem is $(u_{01}(x), u_{02}(x))^T=(2x, x+1)^T$. Eventhough $u_{01}(x)$ coincides with $u_1(x)$ at the boundary points, $u_{02}(0) \ne u_2(0)$ implies that $\varepsilon_2$-layer may occur at $x=0$ in both the solution components $u_1$ and $u_2$ . For $N=1024,\; \varepsilon_1=5^{-6},\; \varepsilon_2=2^{-6}$, the plots of the approximate solution components of (\ref{ex21}) - (\ref{ex23})  shown in Figures 2 and 3 ensure the foresaid layer patterns.

\begin{figure}[h]
\includegraphics[width=0.33\textwidth, angle=270]{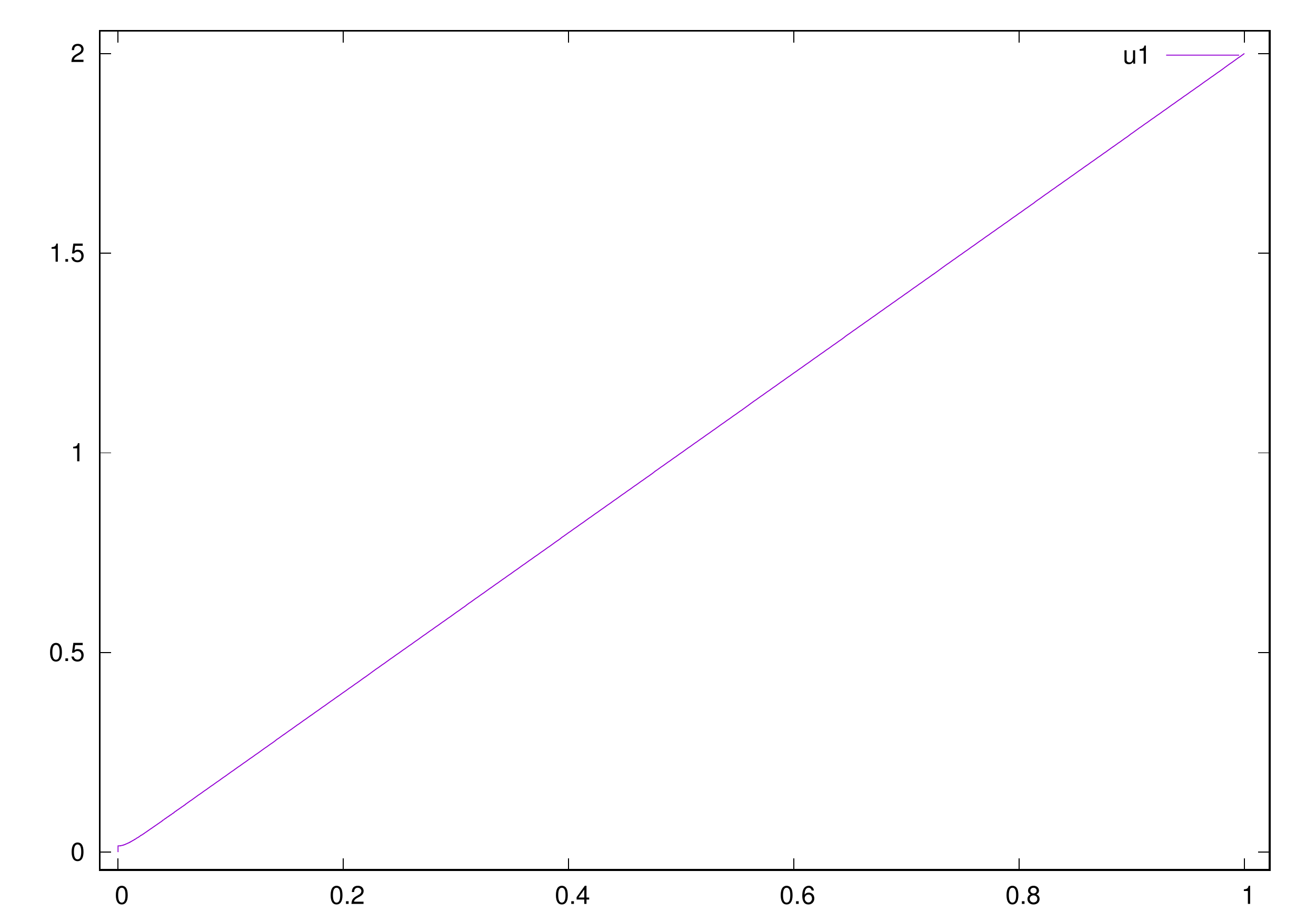}\hfill
\includegraphics[width=0.33\textwidth, angle=270]{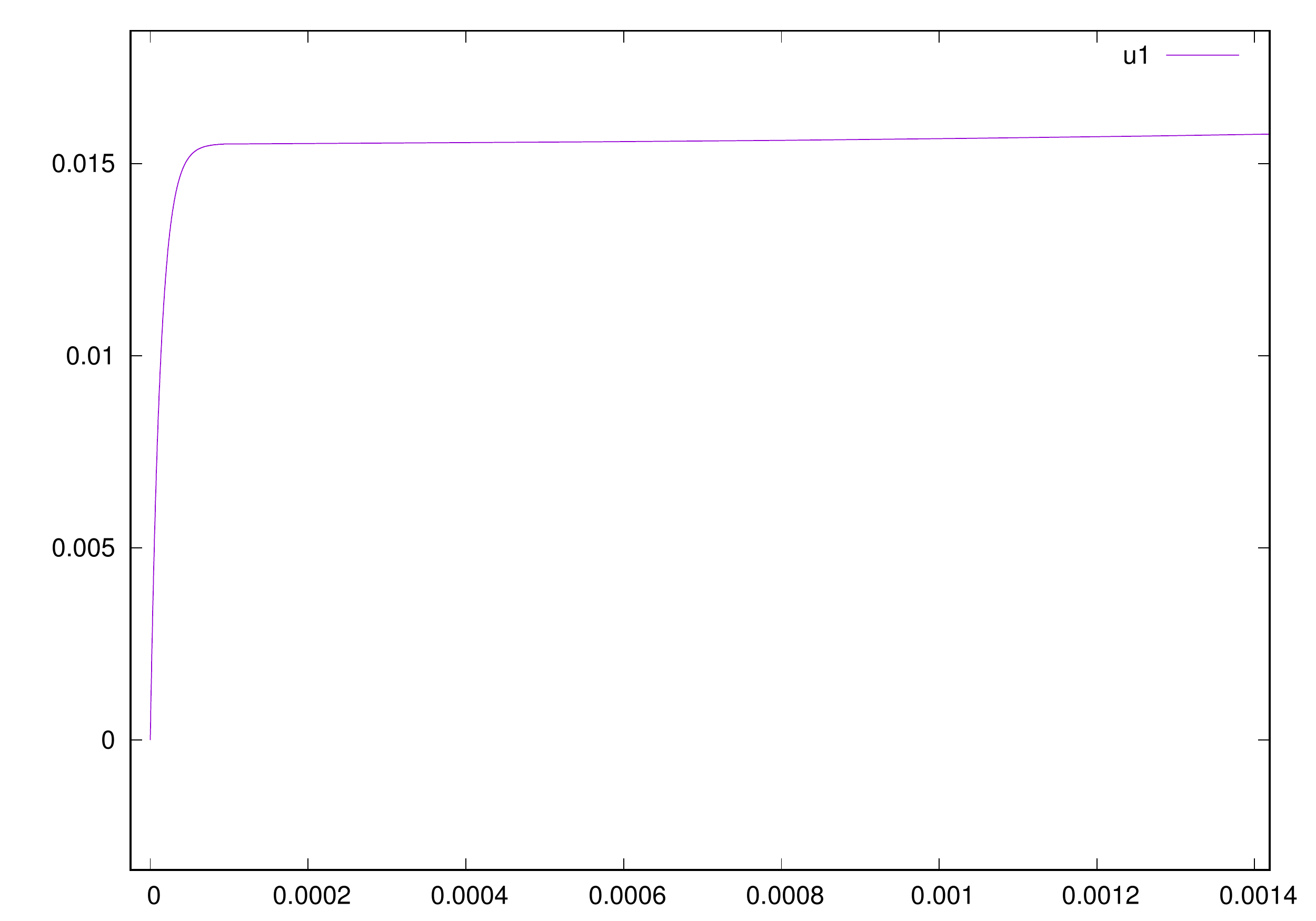}\\
u1 \hspace{4.5cm} u1 near x=0
\caption{Approximation of solution component u1 of Example~\ref{ex2}.}
\end{figure}
\begin{figure}[h]
\includegraphics[width=.5\textwidth, angle=270]{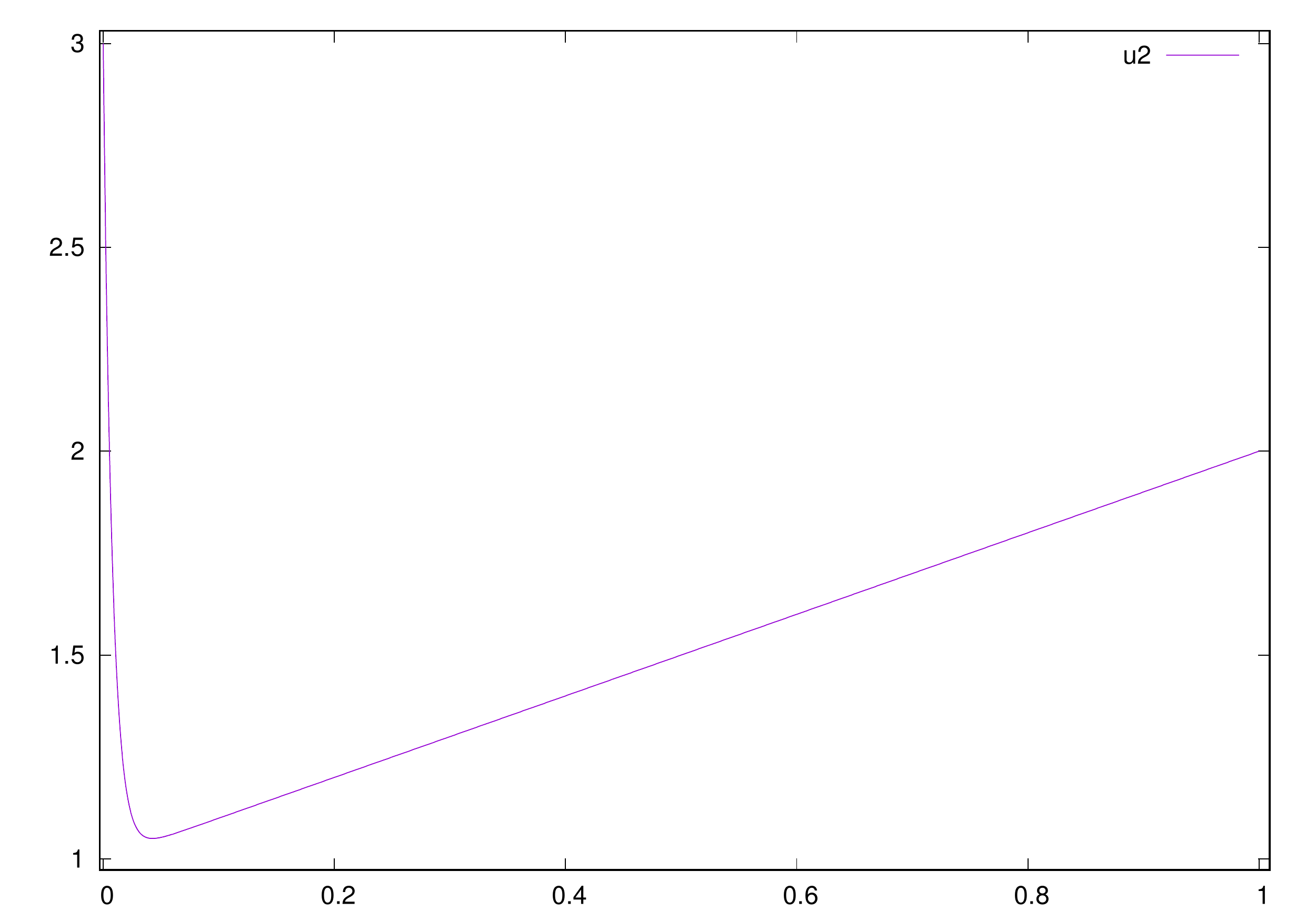}
\caption{Approximation of solution component u2 of Example~\ref{ex2}.}
\label{f22}
\end{figure}

\begin{example}\label{ex3}
Consider the boundary value problem for the system of convection diffusion equations on (0,1)
\begin{eqnarray}\label{ex31}
\varepsilon_1 u_1^{\prime \prime}(x)+u_1^\prime (x)-2u_1(x)+u_2(x)=-3(x-1),\hspace{1.5cm}\\
\label{ex32}
\varepsilon_2 u_2^{\prime \prime}(x)+(1+x)u_2^\prime (x)+xu_1(x)-(2x+1)u_2(x)=-2x,\\ \label{ex33}
 \text{with}\; u_1(0)  =1,\;u_2(0)=1,\; u_1(1)  =2,\;u_2(1)=2.\hspace{1.3cm}
\end{eqnarray}
\end{example}
Solution of the reduced problem is $(u_{01}(x), u_{02}(x))^T=(2x, x+1)^T$. Since,  $u_{01}(0) \ne u_1(0)$ and $u_{02}(0) = u_2(0)$, $\varepsilon_1$-layer is expected near $x=0$ only in the solution component $u_1$. For $N=1024,\; \varepsilon_1=5^{-4},\; \varepsilon_2=2^{-4}$, the plots of the approximate solution components of (\ref{ex31}) - (\ref{ex33}) shown in Figure 4 ensures the foresaid layer patterns.

\begin{figure}[t]
  \includegraphics[width=0.33\textwidth, angle=270]{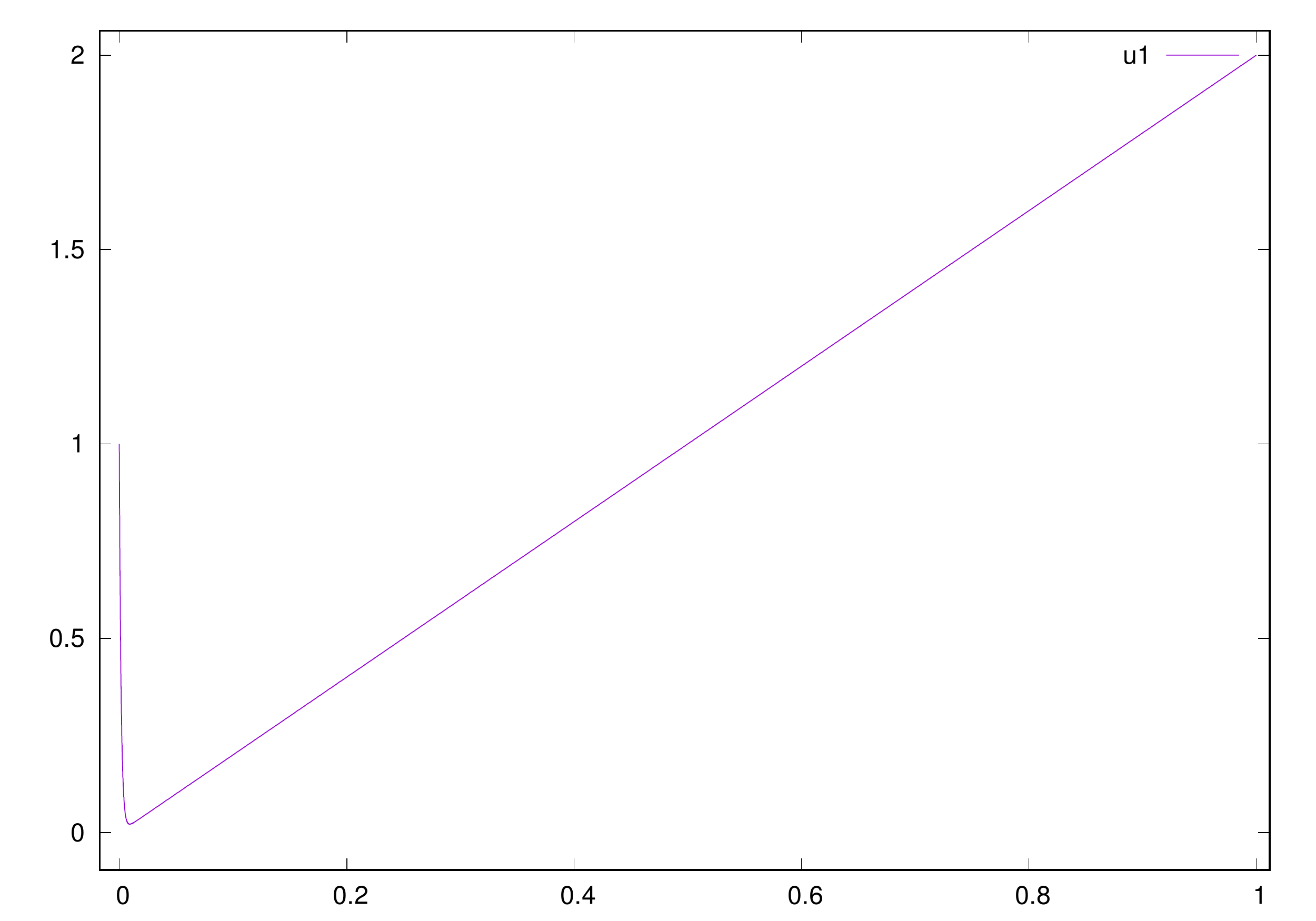}\label{f31}
  \hfill
  \includegraphics[width=0.33\textwidth, angle=270]{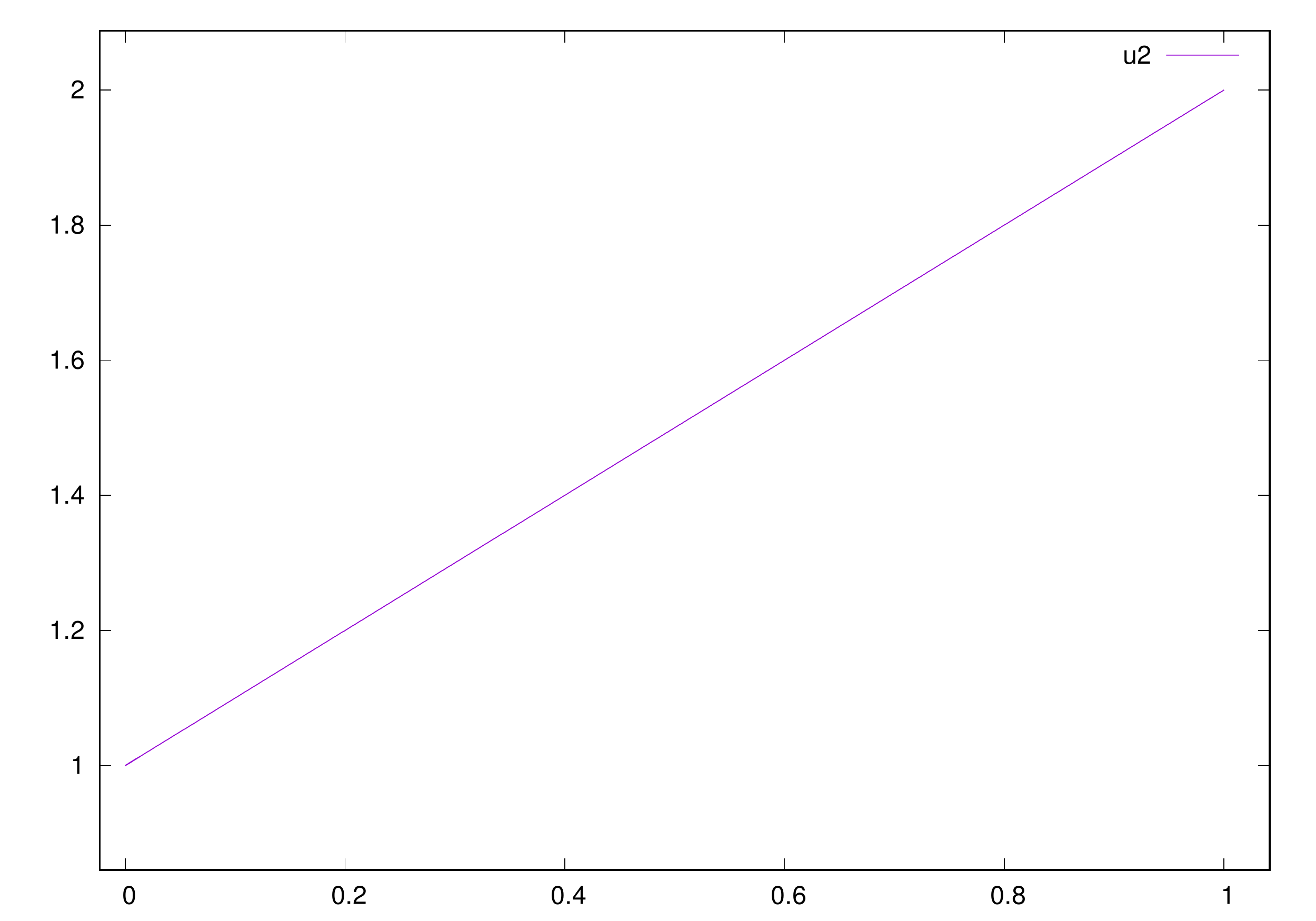}\label{f32}
  \caption{approximation of solution components of Example~\ref{ex3}.}
\end{figure}
\section*{Acknowledgement}
The first author wishes to acknowledge the financial support extended through Junior Research Fellowship by the University Grants Commission, India, to carry out this research work. Also the first and the third authors thank the Department of Science \& Technology, Government of India for the support to the Department through the DST-FIST Scheme to set up the Computer Lab where the computations have been carried out.\\


\begin{thebibliography}{15}
\bibitem{MOS96} Miller, J. J. H., O'Riordan, E. and Shishkin, G.I.:
 Fitted numerical methods for singular perturbation problems, World Scientific Publishing Co., Singapore (1996).
\bibitem{DMS80} Doolan, E. P., Miller, J. J. H. and Schilders, W. H. A.:
Uniform numerical methods for problems with initial and boundary layers, Boole press, Dublin, Ireland(1980).
\bibitem{FHM00} Farrell, P.A., Hegarty, A., Miller, J.J.H., O'Riordan, E. and Shishkin, G.I.:
Robust computational techniques for boundary layers,  Chapman and Hall/CRC Press, Boca Raton(2000).
\bibitem{NM03}Niall Madden and Martin Stynes:
A uniformly convergent numerical for a coupled system of two singularly perturbed linear reaction-diffusion problems, IMA Journal of Numerical Analysis, 23, 627-644 (2003).
\bibitem{LM09} Linss, T. and Madden, N.:
Layer-adapted meshes for a linear system of coupled singularly perturbed reaction-diffusion problems, IMA Journal of Numerical Analysis, 29, 109-125 (2009).
\bibitem{PVM10} Paramasivam, M., Valarmathi, S., and Miller, J.J.H.:
Second order parameter-uniform convergence for a finite difference method for a singularly perturbed linear reaction-diffusion system,
Math. Commun., 15(2), 587-612 (2010).
\bibitem{CGL10} Clavero, C., Gracia, J.L. and Lisbona, F.J.:
An almost third order finite difference scheme for singularly perturbed reaction-diffusion system, Journal of Computational \& Applied Mathematics, 234(8), 2501-2515(2010).
\bibitem{BO04} Bellow, S. and O'Riordan, E.:
 A parameter robust numerical method for a system of two singularly perturbed convection-diffusion equations, Applied Numerical Mathematics 51(2-3): 171-186(2004).
\bibitem{OM09}O'Riordan, E. and  Martin Stynes:
Numerical analysis of a strongly coupled system of two singularly perturbed convection-diffusion problems, Adv. Comput. Math., 30, 101-121(2009).
\bibitem{ZC04} Zhongdi Cen.: 
Parameter-uniform finite difference scheme for a system of coupled singularly perturbed convection-diffusion equations,
International Journal of Computer Mathematics, 82(2), 177-192 (2005).
\bibitem{L07} Linss, T.:
Analysis of an upwind finite difference scheme for a system of coupled singularly perturbed convection-diffusion equations, Computing, 79(1), 23-32(2007).

\end{thebibliography}
\end{document}